\documentclass[12pt]{article}  
\usepackage{amsmath}
\usepackage{amssymb}
\usepackage{theorem}
\usepackage{euscript}
\usepackage{pstricks}
\usepackage{authblk}
\usepackage{verbatim}
\usepackage{graphics}
\usepackage{graphicx}
\usepackage{tikz}
\usepackage{amscd}

\usepackage{hyperref}
\hypersetup
{
	colorlinks,
	citecolor=blue,
	filecolor=green,
	linkcolor=green,
	urlcolor=blue
}
\hypersetup{linktocpage}

\usepackage{color}

\newtheorem{theorem}{Theorem}[section]
\newtheorem{lemma}[theorem]{Lemma}
\newtheorem{corollary}[theorem]{Corollary}
\numberwithin{equation}{section}

\theoremstyle{definition}
 
\theoremstyle{remark}
\newtheorem{remark}[theorem]{Remark}
\newcommand{\brac}[1]{\left(#1\right)}

\newcommand{\bk}{{\boldsymbol{k}}}

\newcommand{\br}{{\boldsymbol{r}}}
\newcommand{\bs}{{\boldsymbol{s}}}

\newcommand{\bx}{{\boldsymbol{x}}}

\newcommand{\rd}{{\rm d}}

\newcommand{\sign}{{\rm sign\,}} 

\def\ZZd{{\mathbb Z}^d}
\def\IId{{\mathbb I}^d}
\def\ZZ{{\mathbb Z}}

\def\RR{{\mathbb R}}
\def\RRd{{\mathbb R}^d}

\def\NN{{\mathbb N}}
\def\NNd{{\NN}^d}

\def\NN{{\mathbb N}}
\def\RR{{\mathbb R}}

\def\IId{{\mathbb I}^d}

\def\NNd{{\mathbb N}^d}
\def\RRd{{\mathbb R}^d}

\def\ZZd{{\mathbb Z}^d}

\def\Hh{{\mathcal H}}

\def\Ll{{\mathcal L}}

\def\ZZ{{\mathbb Z}}
\def\NN{{\mathbb N}}
\def\RR{{\mathbb R}}

\def\IId{{\mathbb I}^d}

\def\NNd{{\mathbb N}^d}
\def\RRd{{\mathbb R}^d}

\def\Wpmix{W^\alpha_p(\IId)}

\def\Wpgamma{W^\alpha_p(\RRd,\gamma)}
\def\Wap{W^\alpha_p}
\def\Wa{W^\alpha_2(\RRd,\gamma)}
\def\Lqgamma{L_q(\RRd,\gamma)}

\newcommand{\norm}[2]{\left\|{#1}\right\|_{#2}}

\title{\sffamily {Pseudo s-Numbers of Embeddings of Gaussian Weighted Sobolev Spaces}}

\author{Van Kien Nguyen}
\affil{Department of Mathematical Analysis, University of Transport and Communications
	\protect\\	No.3 Cau Giay Street, Lang Thuong Ward, Dong Da District,
	Hanoi, Vietnam
	\protect\\
	Email: kiennv@utc.edu.vn}

\date{\today}
\begin{document}
\maketitle

\begin{abstract}
In this paper, we study the approximation problem for functions in the Gaussian-weighted Sobolev space  $W^\alpha_p(\mathbb{R}^d, \gamma)$ of  mixed smoothness $\alpha \in \mathbb{N}$ with error measured in the Gaussian-weighted space $L_q(\mathbb{R}^d, \gamma)$. We obtain the exact asymptotic order of pseudo $s$-numbers for the cases $1 \leq q< p < \infty$ and $p=q=2$. Additionally, we also obtain an upper bound  and a lower bound for pseudo $s$-numbers of the embedding of $W^\alpha_2(\mathbb{R}^d, \gamma)$ into $L_{\infty}^{\sqrt{g}}(\RRd)$. Our result is an extension of that obtained in Dinh Dũng and Van Kien Nguyen (IMA Journal of Numerical Analysis, 2023) for approximation and Kolmogorov numbers.  
	
	\medskip
	\noindent
	{\bf Keywords and Phrases}:  Gaussian-weighted Sobolev space of mixed smoothness; pseudo $s$-numbers; Asymptotic order of convergence. 
	
	\medskip
	\noindent
	{\bf MSC (2020)}:   41A25; 41A46, 46E35
	
\end{abstract}

\section{Introduction}
 \label{Introduction}

Recently  there  is  an  increasing  interest  in  the  study  of multivariate numerical integration and    approximation   in  the  context of Gaussian-weighted Sobolev  spaces of mixed smoothness \cite{IKLP2015, IL2015,DILP18,DN23}. This  is motivated by  the  fact 
that a number of real-world problems in finance, physics, quantum chemistry and machine learning are modeled on function 
spaces on high-dimensional domains, and often  the  functions  to be integrated or approximated have some Sobolev regularity equipped Gaussian measure.

In recent paper \cite{DN23}, Dinh Dũng and the author studied the 
linear  approximation and sampling recovery for  functions from Gaussian-weighted  Sobolev spaces  $\Wpgamma$ in the Gaussian-weighted space $L_q(\RRd,\gamma)$. The asymptotic optimality in terms 
of the Kolmogorov, linear and sampling numbers of the approximation was obtained for the cases $1\le q < p <\infty$ and $p=q=2$. The problem of multivariate numerical integration for functions in $\Wpgamma$ has been studied in \cite{IKLP2015, IL2015,DILP18,DN23}. An asymptotically optimal quadrature of numerical integration was constructed recently in \cite{DN23} based on an assembling method.

In  the  present  paper we  deal with  the approximation  problem. We study the asymptotic behavior of pseudo $s$-numbers of the embedding of $W^\alpha_p(\mathbb{R}^d, \gamma)$ into  $L_q(\mathbb{R}^d, \gamma)$. Our 
approach is based on the decomposition technique developed from \cite{DN23} and on tools from the theory of pseudo $s$-numbers, see \cite{Pie80B}.  We show that in the cases $1 \leq q< p < \infty$ and $p=q=2$, the  exact asymptotic order of  pseudo $s$-numbers can be obtained. In particular, if $1\leq q<p<\infty$, for injective and additive pseudo $s$-numbers we get
$$
s_n\big(I_\gamma:\Wpgamma\to L_q(\RRd,\gamma)\big) \asymp s_n\big(I:\tilde{W}_p^\alpha(\IId) \to \tilde{L}_q(\IId)\big),\ \ n\to \infty\,.
$$
Here $I$ is the embedding from periodic Sobolev space with mixed smoothness $\tilde{W}_p^\alpha(\IId)$ into the Lebesgue space $\tilde{L}_q(\IId)$ on the torus $\IId:=\big[-\frac{1}{2},\frac{1}{2}\big]^d$.

Additionally, in this paper we also obtain an upper bound and a lower bound for asymptotic behavior of pseudo $s$-numbers of the embedding of $W^\alpha_2(\mathbb{R}^d, \gamma)$ into $L_{\infty}^{\sqrt{g}}(\RRd)$. Note here that we do not have  a continuous embedding of $\Wpgamma$ into $L_\infty(\RRd,\gamma)$ if $1\leq p<  \infty$.  The space $L_{\infty}^{\sqrt{g}}(\RRd)$ is the collection of all functions $f$ on $\RRd$ such that $\big|f(\bx)\sqrt{g(\bx)}\big|$ is bounded. Here $g(\bx)$ is the density of the standard Gaussian measure on $\RRd$.  
In this context we obtain
\begin{equation*}
	\begin{aligned}
		n^{-\frac{\alpha}{2}-\frac{d}{4}} (\log n)^{(\frac{\alpha}{2}+\frac{d}{4})(d-1)} & \ll
		s_n\big(I_\gamma: W_2^\alpha(\RRd,\gamma) \to L_{\infty}^{\sqrt{g}}(\RRd)\big)
		\\
		&\ll n^{-\frac{\alpha}{2}-\frac{1}{12}+\frac{1}{2}}(\log n)^{(\frac{\alpha}{2}+\frac{1}{12})(d-1)}\,,\qquad \ n\to \infty.
	\end{aligned}
\end{equation*}

The paper is organized as follows. In Section \ref{sec-pseudo}, we recall the notion of pseudo $s$-numbers and some particular pseudo $s$-numbers.  Section \ref{Approximation} is devoted to the proof of the asymptotic order of pseudo $s$-numbers of the embedding $I_\gamma$  for the case $1\leq q < p<\infty$. In Section \ref{sec-p=2}, we study approximation problem for functions in $W_2^\alpha(\RRd,\gamma)$ with error measured in $L_2(\RRd,\gamma)$ or $L_\infty^{\sqrt{g}}(\RRd)$.

\noindent
{\bf Notation.} 
The letter $d$ is always reserved for
the underlying dimension of $\RR^d$, $\NN^d$, etc. Vectors in $\RRd$  are denoted by boldface
letters. For $\bx \in \RR^d$, $x_i$ denotes the $i$th coordinate, i.e., $\bx := (x_1,\ldots, x_d)$.  For a real number $a$ we denote by $\lfloor a \rfloor$ the greatest integer not larger than $a$.
For the quantities $A_n$ and $B_n$ depending on 
$n$ in an index set $J$  
we write  $A_n \ll B_n$  
if there exists some constant $C >0$ independent of $n$ such that 
$A_n \leq CB_n$ for all $n \in J$, and  
$A_n \asymp B_n$ if $A_n  \ll B_n $
and $B_n  \ll A_n $. General positive constant (may depend on parameters) is denoted by $C$. Values of constant $C$ is not specified and may be different in various places. 
\section{Pseudo s-numbers}\label{sec-pseudo}
Let us first recall the definition of pseudo $s$-numbers following Pietsch \cite[Section 12.1.1]{Pie80B}. In the paper \cite{Pie74}, see also \cite[Chapter 11]{Pie80B}, Pietsch introduced the notion of $s$-number. However, dyadic entropy numbers, see the definition below, do not belong to the class of $s$-numbers. To incorporate also dyadic entropy numbers into the framework, Pietsch introduced the notion of pseudo $s$-numbers. 

Let $X,Y,X_0,Y_0$ be Banach spaces.
A pseudo $s$-number is a map $s$ assigning to every linear operator $T\in \mathcal L(X,Y)$ a scalar sequence $(s_n(T))_{n\in \NN}$ such that the following conditions are satisfied:
\begin{enumerate}
	\item[(a)] $\|T\|=s_1(T)\geq s_2(T)\geq\ldots\geq 0 $;
	\item[(b)]$ s_{n}(S+T)\leq s_n(S)+ \|T\| $ for all $S\in \mathcal L(X,Y)$ and $m,n\in \NN\, $;
	\item[(c)] $s_n(BTA)\leq \|B\| \, \cdot \, s_n(T) \, \cdot \, \|A\|$ for all $A\in \mathcal L(X_0,X)$, $B\in \mathcal L(Y,Y_0)$\,.
\end{enumerate}
 A pseudo $s$-number is called additive if it satisfies
 \begin{enumerate}
 	\item[(b')]$ s_{n+m-1}(S+T)\leq s_n(S)+ s_m(T) $ for all $S\in \mathcal L(X,Y)$ and $m,n\in \NN\, $.
 \end{enumerate}
Some of the popular pseudo $s$-numbers are listed below:
	\begin{enumerate}
\item [(i)] The $n$th approximation number of the linear operator $T$ is defined as
$$
a_n(T):=\inf\{\|T-A\|: \ A\in \mathcal L(X,Y),\ \ \text{rank} (A)<n\}\, , \qquad n \in \NN\, . 
$$
\item [(ii)] The $n$th Kolmogorov number of the linear operator $T \in \mathcal L(X,Y)$ is defined as 
\begin{equation*}
	d_n(T)= \inf_{L_{n-1}}\sup_{\|x\|_X\leq 1}\inf_{y\in L_{n-1}}\|Tx-y\|_Y. \label{def1}
\end{equation*}
Here the outer supremum is taken over all linear subspaces  $L_{n-1}$ of dimension ($n-1$)  in $Y$. 
\item [(iii)] The $n$th Gelfand number of the linear operator $T \in \mathcal L(X,Y)$ is defined as 
$$ c_n (T) := \inf\Big\{\|\, T\, J_M^X\, \|: \ {\rm codim\,}(M)< n\Big\},$$
where $J_M^X:M\to X$ refers to the canonical injection of $M$ into $X$. 
\item [(iv)] The $n$th Weyl number of $T$ is defined as 
		$$ x_n(T):=\sup\{a_n(TA):\ A\in \mathcal L(\ell_2,X),\ \|A\|\leq 1\}\, , \qquad n \in \NN\, .$$
	\item [(v)] The $n$th (dyadic) entropy
		number of $T$ is defined as
		$$
		e_n (T):=\inf\{ \varepsilon
		>0: T(B_X) \text{ can be covered by } 2^{n-1}
		\text{ balls in } Y \text{ of radius } \varepsilon\}\, ,
		$$
		where $B_X:= \{x \in X: \: \|x\|_X \le 1\}$ denotes the closed
		unit ball of $X$.
	\item [(vi)]  The $n$th Bernstein number of the operator $T\in \Ll(X,Y)$  is defined as
	$$ b_n(T)= \sup_{L_n}\inf_{\substack{x\in L_n
			\\ x\not =0}} \dfrac{\|Tx\|_Y}{\| x\|_X} ,$$
	where the supremum is taken over all subspaces $L_n$ of $X$ with dimension $n$.
\end{enumerate}
Note that approximation, Kolmogorov, Gelfand, Weyl and entropy numbers are additive pseudo $s$-numbers, see \cite[Chapters 11 and 12]{Pie80B}. Bernstein number is not an additive pseudo $s$-number \cite{Pie08}.

It is well-known that approximation number is the largest pseudo $s$-number in the set $\{a,b,c,d,x\}$, see \cite{Pie74,CaSt90B,Pie80B}. We also have the following inequalities
\begin{equation}\label{eq-inequality1}
b_n(T)\leq \min\{c_n(T),d_n(T)\}, 
\end{equation}
see \cite{Pie74} and
\begin{equation}\label{eq-inequality2}
 e_n(T)\leq a_n(T),\ \ \ b_n(T)\leq  2\sqrt{2}e_n(T).
\end{equation}
The first inequality in \eqref{eq-inequality2} can be found in  \cite{CaSt90B}, the second one was proved in \cite{Ng16}. Moreover, if $x_n(T)\asymp n^{-a}(\log n)^{b}$ with $a>0,\ b\geq 0$ then 
\begin{equation}\label{eq-inequality3}
b_n(T)\leq Cx_n(T)
\end{equation}
which is a consequence of \cite[Lemma 2]{Pie08}. For a proof, we refer the reader to  \cite[Corollary 3.2]{Ng15}.
Similarly, if $e_n(T)\asymp n^{-a}(\log n)^{b}$ we get
\begin{equation}\label{eq-inequality4}
	e_n(T)\leq Cd_n(T),
\end{equation}
see \cite{Tem98}. 

We recall that a metric injection $J$ of a Banach space $Y$ into a Banach space 
$\tilde{Y}$ is characterized by the property 
$$
\|Jy\|_{\tilde{Y}}=\|y\|_Y,\ \text{for all } y\in Y.
$$
A  pseudo $s-$number  is called weakly injective if,  given any metric  injection  $J\in \Ll(Y,\tilde{Y})$ we have
\begin{equation}\label{eq-injective}
s_n(T)\asymp s_n(JT).
\end{equation}
In case $s_n(T)= s_n(JT)$ it is called injective. We know that Gelfand, Bernstein, and Weyl numbers are injective pseudo $s$-numbers, while entropy number is weakly injective. Approximation and Kolmogorov numbers are not injective. 
 \section{The case $1\leq q<p<\infty$}
 \label{Approximation}
 In this section, we study  the asymptotic behavior of pseudo $s$-numbers of the embedding of Gaussian-weighted  Sobolev spaces $W^\alpha_p(\mathbb{R}^d, \gamma)$ of  mixed smoothness $\alpha \in \mathbb{N}$ into the Gaussian-weighted space $L_q(\mathbb{R}^d, \gamma)$ in the case $1 \leq q< p < \infty$. 
 
 Let us first introduce Gaussian-weighted Sobolev spaces of mixed smoothness.	
 Denote 
 $$
 g(\bx):=(2\pi)^{-d/2} \exp\brac{-|\bx|^2/2},\ \ \bx\in \RRd 
 $$
 and   $\gamma (\rd \bx) := g(\bx) \rd\bx$  the $d$-dimensional standard Gaussian measure on $\RRd$  with the density $g(\bx)$. For  $1\leq p<\infty$ and  a Lebesgue measurable set $\Omega\subset\RRd$, we define the Gaussian-weighted space  $L_p(\Omega,\gamma)$ to be the set of all functions $f$ on $\Omega$ such that the norm
	$$
	\|f\|_{L_p(\Omega,\gamma)} : = \bigg( \int_\Omega |f(\bx)|^p \gamma(\rd \bx)\bigg)^{1/p}
	=
	\bigg( \int_\Omega |f(\bx)|^p g(\bx) \rd \bx\bigg)^{1/p} \ 
	$$
	is finite. 
	For $\alpha \in \NN$, we define the Gaussian-weighted  space $\Wap(\Omega,\gamma)$ to be the normed space of all functions $f\in L_p(\Omega,\gamma)$ such that the generalized partial
	 derivative $D^\br f$ of order $\br$  belongs to $L_p(\Omega,\gamma)$ for all $\br\in \NN_0^d$ satisfying $|\br|_\infty\leq \alpha$. The norm of a  function $f$ in this space 
	is defined by
	\begin{align} \label{W-Omega}
		\|f\|_{\Wap(\Omega,\gamma)}: = \Bigg(\sum_{|\br|_\infty \leq \alpha} \|D^\br f\|_{L_p(\Omega,\gamma)}^p\Bigg)^{1/p}.
	\end{align}
	The space $\Wap(\Omega)$ is defined  as the classical Sobolev space with mixed smoothness by replacing $L_p(\Omega,\gamma)$ with $L_p(\Omega)$ in \eqref{W-Omega}, where as usual,   $L_p(\Omega)$ denotes the Lebesgue space of functions on $\Omega$ equipped with the usual $p$-integral norm.



For a fixed   $\theta>1 $ we denote the $d$-cube $\IId_\theta$ by $\IId_\theta := \big[-\frac{\theta}{2}, \frac{\theta}{2}\big]^d$, 
$\IId_{\theta,\bk}:=\bk+\IId_\theta$  for $\bk \in \ZZd$, and
$ f_{\theta,\bk} $ the restriction of $f$ on $\IId_{\theta,\bk}$  for a function $f$ on $\RRd$.  Let  $(\varphi_\bk)_{\bk \in \ZZd}$ be a partition of unity on $\RRd$ satisfying  
\begin{itemize}
	\item[\rm{(i)}] $\varphi_\bk \in C^\infty_0(\RRd)$ and 
	$0 \le \varphi_\bk (\bx)\le 1$, \ \ $\bx \in \RRd$, \ \ $\bk \in \ZZd$;
	\item[\rm{(ii)}] the support of $\varphi_\bk$ is contained in the interior of  $\IId_{\theta,\bk}$, $\bk \in \ZZd$;
	\item[\rm{(iii)}]  $\sum_{\bk \in \ZZd}\varphi_\bk (\bx)= 1$, \ \ $\bx \in \RRd$;
	\item[\rm{(iv)}]  $\norm{\varphi_\bk }{W^\alpha_p(\IId_{\theta,\bk})} \le C_{\alpha,d,\theta}$, \ \ 
	$\bk \in \ZZd$.
\end{itemize}
For a construction of such a system, we refer the reader to  \cite[Chapter VI, 1.3]{Stein1970}.
Observe that if $f \in \Wpgamma$, then we have
\begin{equation*}\label{f_theta,bk}
f_{\theta,\bk}(\cdot+\bk)\varphi_\bk(\cdot+\bk)  \in  {W}^\alpha_p(\IId_{\theta}),
\end{equation*}
and it holds from algebra property of ${W}^\alpha_p(\IId_\theta)$, see \cite{NgS17}, that
 \begin{equation}\label{eq-first}  
	\begin{aligned} 
		\|f_{\theta,\bk}(\cdot+\bk)\varphi_\bk(\cdot+\bk)\|_{{W}^\alpha_p(\IId_\theta)} 
		& 
		\leq C \|f_{\theta,\bk}(\cdot+\bk)\|_{{W}^\alpha_p(\IId_\theta)}   \cdot \|\varphi_{\bk}(\cdot+\bk)\|_{ {W}^\alpha_p(\IId_\theta)} \\
		&
		= C \|f_{\theta,\bk}(\cdot+\bk)\|_{{W}^\alpha_p(\IId_\theta)}   \cdot \|\varphi_{\bk} \|_{ {W}^\alpha_p(\IId_{\theta,\bk})} \\
		&\leq C\|f_{\theta,\bk}(\cdot+\bk)\|_{{W}^\alpha_p(\IId_\theta)}.
	\end{aligned}
\end{equation}
Here in the last estimate we used property (iv) of the system $(\varphi_\bk)_{\bk\in \NN_0^d}$. Note, that when $\bx \in \IId_{\theta}$ we have 
$e^{\frac{|\bx+\bk|^2}{2}}\leq e^{\frac{|\bk + \theta(\sign \bk)/2 |^2}{2}}$,  where $\sign \bk:= \brac{\sign k_1, \ldots , \sign k_d}$ and $\sign x := 1$ if $x \ge 0$, and $\sign x := -1$ otherwise for $x \in \RR$. Therefore
 \begin{equation*}  
	\begin{aligned}  
\|f_{\theta,\bk}(\cdot+\bk)\|_{{W}^\alpha_p(\IId_\theta)}
&=	\Bigg(\sum_{|\br|_\infty \leq \alpha}  \int_{\IId_\theta} |D^\br f_\bk(\bx+\bk)| ^p \rd\bx\Bigg)^{1/p}
	\\	
&=	\Bigg(\sum_{|\br|_\infty \leq \alpha} (2\pi)^{d/2} \int_{\IId_\theta}e^{\frac{|\bx+\bk|^2}{2}}|D^\br f_\bk(\bx+\bk)| ^pg(\bx+\bk)\rd\bx\Bigg)^{1/p}
\\
		& \leq C e^{\frac{|\bk + (\theta\sign \bk)/2 |^2}{2p} }\|f_{\theta,\bk}\|_{ {W}^\alpha_p(\IId_{\bk,\theta},\gamma)}
	\\
	&
	\leq C e^{\frac{|\bk + (\theta\sign \bk)/2 |^2}{2p} }\|f\|_{\Wpgamma}.
	\end{aligned}
\end{equation*}
This and \eqref{eq-first} deduce 
 \begin{equation}  
	\begin{aligned} \label{multipl-algebra2}
	\|f_{\theta,\bk}(\cdot+\bk)\varphi_\bk(\cdot+\bk)\|_{{W}^\alpha_p(\IId_\theta)} 
		\leq C e^{\frac{|\bk + (\theta\sign \bk)/2 |^2}{2p} }\|f\|_{\Wpgamma}.
	\end{aligned}
\end{equation}
We have $-|x|^2\leq \frac{\theta^2}{2}-\big|k-\frac{\theta(\sign k)}{2}\big|^2 $ for $x\in [-\frac{\theta}{2},\frac{\theta}{2}]+k,\ k\in \ZZ$. Hence for any $h\in L_q(\IId_{\theta})$ we obtain
\begin{equation}\label{eq-norm-B}
	\begin{aligned}
	\norm{h(\cdot-\bk)}
	{L_q(\IId_{\theta,\bk}, \gamma)} &= \bigg( (2\pi)^{-d/2}\int_{\IId_{\bk,\theta}}\big|h(\bx-\bk)\big|^q e^{-\frac{|\bx|^2}{2}}\rd\bx\bigg)^{1/q}
	\\
	&\leq C e^{- \frac{|\bk - (\theta\sign \bk)/2 |^2}{2q}}\bigg( \int_{\IId_{\bk,\theta}}\big|h(\bx-\bk)\big|^q \rd \bx\bigg)^{1/q}
\\
&
 \ll e^{- \frac{|\bk - (\theta\sign \bk)/2 |^2}{2q}}
	\norm{h}
	{{L}_q(\IId_{\theta})} .
\end{aligned}
\end{equation}
If $q<p$,  we can choose a fixed $\delta>0$ such that
\begin{equation}  \label{<e^{-delta k}}
	{e^{\frac{|\bk + (\theta \sign \bk)/2 |^2}{2p} 
			-\frac{|\bk - (\theta\sign \bk)/2 |^2}{2q}}}
	\leq 
	C e^{- \delta |\bk|^2}, \ \  \bk\in \ZZd.
\end{equation}

Denote by $\tilde{L}_q(\IId)$ and $\tilde{W}^\alpha_p(\IId)$ the subspaces of  $L_q(\IId) $ and $\Wpmix$, respectively,  of all functions $f$ which can be extended to the whole $\RRd$ as $1$-periodic  functions in each variable (denoted again by $f$). Similarly we defined $\tilde{L}_q(\IId_\theta)$ and $\tilde{W}^\alpha_p(\IId_\theta)$. 
 Our main result in this section reads as follows. 
 
\begin{theorem} \label{thm:approx-general-theta}
	Let $\alpha\in \NN$, $1\le q < p <\infty$ and $a >0$, $b \ge 0$, $  \theta >1$. Let $s$ be a pseudo $s$-number. 
	Assume that  
\begin{equation}\label{eq-assumption}
s_n\big(I: \tilde{W}^\alpha_p(\IId)\to \tilde{L}_q(\IId)\big)\asymp n^{-a} (\log n)^b\,, \  n\to \infty.
\end{equation}
\begin{enumerate}
	\item If $s$ is weakly injective in the sense \eqref{eq-injective}, then
	\begin{equation*}\label{A_m-Error-a,b}
		s_n\big(I_\gamma:\Wpgamma\to L_q(\RRd,\gamma)\big) \gg   n^{-a} (\log n)^b \,, \ n\to \infty .
	\end{equation*}
	\item If $s$ is additive, then
	\begin{equation*} 
		s_n\big(I_\gamma:\Wpgamma\to L_q(\RRd,\gamma)\big) \ll   n^{-a} (\log n)^b \,, \ n\to \infty.
	\end{equation*}
\end{enumerate}
\end{theorem}

\begin{proof} 
{\bf Step 1.} {\it Lower bound}. Consider the commutative diagram 
\[
\begin{CD}
	\Wpgamma @ > I_{\gamma} >> L_q(\RRd,\gamma) \\
	@AA I_1 A @VV I_2 V \\
	\tilde{W}^\alpha_p(\IId) @ > I >> L_q(\IId) \,,
\end{CD}
\]
where $I_1$ is the embedding operator and $I_2$ is the restriction of functions on $\RRd$ onto $\IId$. 
If $f$ is a $1$-periodic function on $\RRd$ and $f\in  \tilde{W}^\alpha_p(\IId)$, then
\begin{equation*}\label{eq-low01}
\begin{aligned}
	\|f\|_{\Wpgamma}&= \Bigg((2\pi)^{-d/2}\sum_{|\br|_\infty \leq \alpha} \int_{\RRd} |D^\br f(\bx)|^p e^{-\frac{|\bx|^2}{2}}\rd \bx\Bigg)^{1/p}
	\\
	& = (2\pi)^{-\frac{d}{2p}}\Bigg(\sum_{|\br|_\infty \leq \alpha} \sum_{\bk\in \ZZd} \int_{\IId} |D^\br f(\bx+\bk)|^p e^{-\frac{|\bx+\bk|^2}{2}}\rd \bx\Bigg)^{1/p}
	\\
	& \ll \Bigg(\sum_{|\br|_\infty \leq \alpha}  \int_{\IId} |D^\br f(\bx)|^p \rd \bx \sum_{\bk\in \ZZd}e^{-\frac{|\bk - (\sign \bk)/2 |^2}{2}}\Bigg)^{1/p}
	\\
	&
	\ll 
	\|f\|_{\tilde{W}^\alpha_p(\IId)}
\end{aligned}
\end{equation*}
which implies that $\|I_1\|\ll 1$. For a function $h\in L_q(\RRd,\gamma)$ we have
\begin{equation*}\label{eq-low02}
\|h\|_{L_q(\IId)}=\bigg((2\pi)^{\frac{d}{2}}\int_{\IId}|h(\bx)|^qe^{\frac{|\bx|^2}{2}}g(\bx)\rd \bx\bigg)^{1/q} \leq (2\pi)^{\frac{d}{2q}}e^{\frac{d}{8q}}\|h\|_{\Lqgamma}.
\end{equation*}
This deduces $\|I_2\|\ll 1$. By the ideal property   $(c)$ of pseudo $s$-numbers we get
\begin{equation} \label{eq-low}
s_n(I)= s_n(I_2 I_\gamma I_1) \leq \|I_1\| \cdot s_n(I_\gamma)\cdot \|I_2\|\ll s_n(I_\gamma).
\end{equation}	
Since $s$ is an injective pseudo $s$-number in the weak sense \eqref{eq-injective}, we get
$$
s_n\big(I: 	\tilde{W}^\alpha_p(\IId)\to L_q(\IId)\big)\asymp s_n\big(I: 	\tilde{W}^\alpha_p(\IId)\to \tilde{L}_q(\IId)\big)\asymp n^{-a} (\log n)^b.
$$
This and \eqref{eq-low} give the lower bound.
\\
\noindent
{\bf Step 2.} {\it Upper bound}. First note, that \eqref{eq-assumption} induces
\begin{equation}\label{eq-induce}
s_n\big(I_\theta: \tilde{W}^\alpha_p(\IId_\theta)\to \tilde{L}_q(\IId_\theta)\big)\asymp n^{-a} (\log n)^b\,, \  n\to \infty.
\end{equation}
  We define	for $n\in \NN$,
\begin{equation*} \label{xi-int}	
	\xi_n =  \sqrt{\delta^{-1} 2 a(\log n)}\,,
\end{equation*}
and for $\bk \in \ZZd$ with $\ |\bk|< \xi_n$
\begin{equation*} \label{n_bk}
	n_{\bk}= 
	 \lfloor \varrho n  e^{-\frac{\delta}{2 a}|\bk|^2}  \rfloor,
\end{equation*}
where $\varrho := 2^{-d} \big(1 - e^{-\frac{\delta}{2 a}}\big)^{d}$. We have
\begin{align} 	\label{<n2}
	\sum_{|\bk|< \xi_n}  n_\bk  \le n.
\end{align} 
Indeed,
\begin{equation*}
	\begin{aligned}
		\sum_{|\bk| < \xi_n}n_\bk 
		&	\leq 
		\sum_{|\bk|< \xi_n} \varrho n  e^{-\frac{\delta}{2\alpha}|\bk|^2}
		\leq 
		2^{d} \varrho	n  \sum_{s=0}^{\lfloor \xi_n \rfloor} \binom{s+d-1}{d-1}e^{-\frac{\delta}{2 a}s^2}
		\\
		&	\leq   2^{d}\varrho	n \sum_{s=0}^{\infty} \binom{s+d-1}{d-1}e^{-\frac{\delta}{2 a}s} \leq n, 
	\end{aligned}
\end{equation*}
where in the last estimate we  used the well-known formula
\begin{equation*}\label{eq-auxilary-01}
	\sum_{j=0}^\infty x^j\binom{j+k}{k}=(1-x)^{-k-1}, \ k\in \NN_0, \ x\in (0,1).
\end{equation*}
Let $I_{\theta,\bk}: \Wpgamma \to L_q(\RRd,\gamma)$ be the operator defined as
$$
I_{\theta,\bk}(f):= f\varphi_\bk,  \ f\in \Wpgamma.
$$
By property (iii) of the system $(\varphi_\bk)_{\bk\in \ZZd}$ we have
$
	f
	=  \sum_{\bk \in \ZZd}f_{\theta,\bk}\varphi_\bk 
$
which implies $I_\gamma=\sum_{\bk\in \ZZd}I_{\theta,\bk}$. 
In view of the property (b') of additive pseudo $s$-numbers and \eqref{<n2} we get
\begin{align} \label{f-A_n'^gf}
	s_n(I_\gamma) 
		&\leq
		\sum_{|\bk|< \xi_n} 
		s_{n_\bk}(I_{\theta,\bk}) 
+ \sum_{|\bk|\geq \xi_n} 	\norm{I_{\theta,\bk}}{}. 
\end{align}
For every $\bk\in \NNd_0$, consider the commutative diagram 
\[
\begin{CD}
	\Wpgamma @ > I_{\theta,\bk} >> L_q(\RRd,\gamma) \\
	@VV A_\bk V @AA B_\bk A\\
	\tilde{W}^\alpha_p(\IId_{\theta}) @ > I_{\theta} >> \tilde{L}_q(\IId_{\theta}) \,,
\end{CD}
\]
where
$$
\begin{cases}
A_\bk(f):=f_{\theta,\bk}(\cdot+\bk)\varphi_\bk(\cdot+\bk), & \ f\in \Wpgamma
\\
B_\bk(h):=h(\cdot-\bk),& \  h\in \tilde{L}_q(\IId_{\theta}).
\end{cases}
$$
From \eqref{multipl-algebra2} and \eqref{eq-norm-B} we deduce that
\begin{equation}\label{eq-norm-A-B}
\|A_\bk\| \leq e^{\frac{|\bk + (\theta \sign \bk)/2 |^2}{2p}}\ \ \text{and} \ \ \ \|B_\bk\| \leq e^{-
	\frac{|\bk - (\theta\sign \bk)/2 |^2}{2q}}.
\end{equation}
By this and the property (c) of pseudo $s$-numbers we obtain
	\begin{align*}
s_{n_\bk}(I_{\theta,\bk}) &=s_{n_\bk}(B_\bk   I_\theta   A_\bk)
\leq \|A_\bk\| \cdot s_{n_\bk}(I_\theta)\cdot \|B_\bk\|
		\\
&
		\ll   e^{\frac{|\bk + (\theta \sign \bk)/2 |^2}{2p}-
				\frac{|\bk - (\theta\sign \bk)/2 |^2}{2q}}
	s_{n_\bk}(I_\theta)  .
	\end{align*}
Using  \eqref{<e^{-delta k}}  and \eqref{eq-induce} we get
	\begin{align*}
s_{n_\bk}(I_{\theta,\bk})  
	&
	\ll  e^{- \delta |\bk|^2}  	\Big( n e^{-\frac{\delta}{2a}|\bk|^2} \Big)^{-a} (\log n)^b=e^{- \frac{\delta}{2} |\bk|^2}  	n^{-a} (\log n)^b,
\end{align*}
which implies
\begin{equation}\label{eq-sum-1}
	\begin{aligned}
	\sum_{|\bk|< \xi_n}	s_{n_\bk}(I_{\theta,\bk}) 
	&\ll \sum_{|\bk|< \xi_n}   e^{- \frac{ \delta}{2} |\bk|^2} 
	n^{-a}  (\log n)^b   
	\ll  n^{-a}  (\log n)^b.   
\end{aligned}
\end{equation}
For a fixed $\varepsilon \in (0,1/2)$, from \eqref{eq-norm-A-B} we have
\begin{align*}
	\sum_{|\bk|\geq \xi_n} 	\norm{I_{\theta,\bk}}{}
		& \leq \sum_{|\bk|\geq \xi_n}\|A_\bk\|\cdot \|I_\theta\|\cdot \|B_\bk\|
		\\
		&\ll \sum_{|\bk|\geq \xi_n} 
		 e^{- \frac{|\bk - (\theta \sign \bk)/2 |^2}{2q}+
				\frac{|\bk + (\theta\sign \bk)/2 |^2}{2p}}
		 \|I_\theta\| 
	\\
	&	\ll  \sum_{|\bk|\geq \xi_n}  e^{- \delta |\bk|^2} 
	 \ll e^{- \delta (1-\varepsilon) \xi_n^2} 
	\\
	&	 =\, e^{-2 a (1-\varepsilon) \log n} 
		\ll  n^{-a}  (\log n)^b .
	\end{align*}
	From the last estimate, \eqref{f-A_n'^gf}, and \eqref{eq-sum-1} we obtain the upper bound.
	\hfill
\end{proof}

Study of asymptotic behavior as well as pre-asymptotic estimate of approximation, Kolmogorov, Gelfand, and entropy numbers has a long history. We refer the reader to the book \cite{DTU18B} for results and historical commments. Asymptotic behavior of Weyl and Bernstein numbers of the embedding $I: \tilde{W}^\alpha_p(\IId)\to \tilde{L}_q(\IId)$ was investigated in \cite{Gal91,Ng16}.
As a consequence of Theorem \ref{thm:approx-general-theta} we have the following result.
\begin{corollary}\label{cor-01} Let $\alpha\in \NN$ and $d\in \NN$. Let $1\leq q<p<\infty$. If $1<q $ then 
 we have $$a_n\big(I_\gamma\big)\asymp d_n(I_\gamma)\asymp c_n(I_\gamma)\asymp e_n(I_\gamma)\asymp n^{-\alpha} (\log n)^{(d-1)\alpha},\ \ n\to \infty;$$
and
$$
x_n(I_\gamma)\asymp b_n(I_\gamma)\asymp n^{-\beta}(\log n)^{(d-1)\beta},\ \ n\to \infty,
$$
where
\begin{equation}\label{eq-beta}
\beta =
\left\{\begin{array}{lll}
	\alpha  &  \text{if}\ \   q<p\leq 2\, ,\ \alpha>0, \\
 \alpha-\frac{1}{p}+\frac{1}{2}  & \text{if}\ \   q\leq 2<p,\ \alpha> \frac{1}{p},\\
 \alpha-\frac{1}{p}+\frac{1}{q}  & \text{if}\ \  2\leq q<p,\ \alpha> \frac{1/q-1/p}{p/2-1},\\
 \frac{\alpha p}{2}  & \text{if}\ \  2\leq q <p\,, \ \alpha< \frac{1/q-1/p}{p/2-1} \text{\rm \ or } q\leq 2<p; \alpha < \frac{1}{p}.\
\end{array}\right.
\end{equation}
When $q=1$, the above results still hold for approximation, Kolmogorov,  Weyl and entropy numbers.
\end{corollary}	

\begin{proof} Let $1<q<p<\infty$. We know that   
$$a_n\big(I\big)\asymp d_n(I)\asymp c_n(I)\asymp e_n(I)\asymp n^{-\alpha} (\log n)^{(d-1)\alpha},\ \ n\to \infty,$$
see  \cite[Theorems 4.3.1, 4.5.1, 6.2.1 and Section 9.6]{DTU18B}
and
$$
x_n(I_\gamma)\asymp b_n(I_\gamma)\asymp n^{-\beta}(\log n)^{(d-1)\beta},\ \ n\to \infty,
$$
where $\beta$ is given in \eqref{eq-beta}, see \cite[Theorems 2.1, 2.3]{Ng16} and \cite{Gal91}. 
Since approximation, Kolmogorov, Gelfand, Weyl and entropy numbers are additive pseudo $s$-numbers, by Theorem \ref{thm:approx-general-theta} we obtain the upper bound for these  pseudo $s$-numbers of the embedding $I_\gamma$. The upper bound for Bernstein numbers follows from the inequality $b_n(I_\gamma)\ll x_n(I_\gamma)$, see \eqref{eq-inequality3}.

By the weakly injective property and Theorem \ref{thm:approx-general-theta} we get lower bound for Bernstein, Gelfand, entropy, and Weyl numbers of $I_\gamma$. The lower bound for approximation and Kolmogorov numbers follows from the inequalities $e_n(I_\gamma)\leq a_n(I_\gamma)$ and $e_n(I_\gamma)\ll d_n(I_\gamma)$, see \eqref{eq-inequality3} and \eqref{eq-inequality4}.

In view of \cite[Theorems 4.3.1, 4.5.1, 6.2.3]{DTU18B} and \cite[Theorem 2.6]{Ng16}, the argument of the case $q=1$ for Kolmogorov, approximation, entropy and Weyl numbers is carried out similarly.
\hfill 
\end{proof}

%


\section{The case $p = 2$}\label{sec-p=2}

 In this section, we study the approximation of functions from Hilbert spaces $\Hh^\alpha$ which is an extension of $W_2^\alpha(\RRd,\gamma)$ from natural smoothness to the real positive smoothness $\alpha$.
For $k\in \NN_0$, the normalized probabilistic Hermite polynomial
$H_k$ of degree $k$ on $\RR$ is defined by
\begin{equation*} 
	H_k(x) 
	:= 
	\frac{(-1)^k}{\sqrt{k!}} 
	\exp\left(\frac{x^2}{2}\right) \frac{\rd^k}{\rd x^k} \exp\left(-\frac{x^2}{2}\right) .
\end{equation*}
For a multi-index $\bk\in \NNd_0$ we define the $d$-variate Hermite
polynomial $H_\bk$ by
\begin{equation*}\label{H_bk}
	H_\bk(\bx) :=\prod_{j=1}^d H_{k_j}(x_j),
	\;\; \bx\in \RRd.
\end{equation*}
It is well-known that the Hermite polynomials $(H_\bk)_{\bk \in \NNd_0}$ constitute an orthonormal basis of the Hilbert space $L_2(\RRd,\gamma)$ (see, e.g.,  \cite[Section 5.5]{Szego1939}). In particular,  every function $f \in L_2(\RRd,\gamma)$ can be represented by the Hermite series 
\begin{equation}\label{H-series}
	f = \sum_{\bk \in \NNd_0} \hat{f}(\bk) H_\bk \ \ {\rm with} \ \ \hat{f}(\bk) := \int_{\RRd} f(\bx)\, H_\bk(\bx)\gamma(\rd \bx) 
\end{equation}
converging in the norm of $L_2(\RRd,\gamma)$. In addition, there holds  Parseval's identity
\begin{equation*}\label{P-id}
	\norm{f}{L_2(\RRd,\gamma)}^2= \sum_{\bk \in \NNd_0} |\hat{f}(\bk)|^2.
\end{equation*}

For $\alpha \in \NN_0$ and $\bk \in \NNd_0$, we define the sequence $\rho_{\alpha}:=(\rho_{\alpha,\bk})_{\bk\in \NNd_0}$, where
\begin{equation*}\label{rho_bk}
	\rho_{\alpha,\bk}: = \prod_{j=1}^d \brac{k_j + 1}^\alpha.
\end{equation*}
The following lemma was proved in \cite{DILP18}. 
\begin{lemma}\label{lemma:N-eq}
	Let $\alpha \in \NN_0$. Then we have that
	\begin{equation*}\label{N-eq}
		\norm{f}{\Wa}^2 \asymp \sum_{\bk \in \NNd_0} \rho_{\alpha,\bk}|\hat{f}(\bk)|^2, \quad    f \in \Wa.
	\end{equation*}
\end{lemma}
By this lemma, we extend the space $W^\alpha_2(\RRd,\gamma)$ to  any  $\alpha > 0$. Denote by $\Hh^\alpha$ the space of all   functions $f \in L_2(\RRd,\gamma)$ represented by the Hermite series \eqref{H-series} for which  the norm
	\begin{equation*}\label{Hh-norm}
	\norm{f}{\Hh^\alpha} := \Bigg(\sum_{\bk \in \NNd_0} \rho_{\alpha,\bk}|\hat{f}(\bk)|^2\Bigg)^{1/2}
\end{equation*}
is finite.
With this definition, we can identify $W^\alpha_2(\RRd,\gamma)$ with $\Hh^\alpha$ for $\alpha \in \NN$ in the sense of equivalent norms.

\begin{theorem} 	\label{theorem:widths:p=q=2}
	Let $\alpha >0$ and $s\in \{a,b,c,d,e,x\}$. Then 
	\begin{equation*}\label{widths:p=q=2}
	s_n\big(I_\gamma:{\Hh}^\alpha\to  L_2(\RRd,\gamma)\big)
		\asymp 
		n^{-\frac{\alpha}{2}} (\log n)^{\frac{(d-1)\alpha}{2}}, \ \ n\to \infty. 
	\end{equation*}
Moreover, if $s\in \{a,b,c,d,x\}$ then
\begin{equation} \label{eq-asymptotic}
	\lim\limits_{n\to \infty} \frac{s_n\big(I_\gamma:{\Hh}^\alpha\to  L_2(\RRd,\gamma)\big)}{n^{-\frac{\alpha}{2}}(\ln n)^{\frac{\alpha(d-1)}{2}}}= 
	\bigg( \frac{1}{(d-1)!}\bigg)^{\alpha/2}\,.
\end{equation}
\end{theorem}
\begin{proof}
It has been proved in \cite{DN23} that 
$$	a_n\big(I_\gamma\big)=	d_n\big(I_\gamma\big)
\asymp 
n^{-\frac{\alpha}{2}} (\log n)^{\frac{(d-1)\alpha}{2}}, \ \ n\to \infty.
$$
Since  ${\Hh}^\alpha$ and $L_2(\RRd,\gamma)$ are Hilbert spaces, there is only one $s$-number of $I_\gamma$, see \cite{Pie74,Pie80B}. This means
\begin{equation}\label{eq-acdx}
a_n\big(I_\gamma\big)=	d_n\big(I_\gamma\big)=c_n\big(I_\gamma\big)=	x_n\big(I_\gamma\big)=b_n\big(I_\gamma\big)\asymp 
n^{-\frac{\alpha}{2}} (\log n)^{\frac{(d-1)\alpha}{2}}, \ \ n\to \infty.
\end{equation}
From \eqref{eq-inequality2} we get $b_n(I_\gamma)\ll e_n(I_\gamma)\leq a_n(I_\gamma)$. This together with \eqref{eq-acdx}  implies the result for entropy numbers. 

We proceed with the proof of \eqref{eq-asymptotic}. Consider the diagram
	\begin{equation*}
	\begin{CD}
		\Hh^\alpha  @ > I_\gamma >> L_2(\RRd,\gamma) \\
		@VV A V @AA B A\\
		\ell_2(\NNd_0) @ > D_{\sqrt{\rho_{\alpha}}} >> \ell_2(\NNd_0) \,, 
	\end{CD}
\end{equation*}
	where the linear operators $A$, 	$D_{\sqrt{\rho_{\alpha}}} $, and $B $ 
are defined as 
\begin{equation*}	\label{ws-12}
	\begin{aligned}
		Af & : =  \big( \sqrt{\rho_{\alpha,\bk}}\hat{f}(\bk)\big)_{\bk\in \NNd_0}\, , \qquad f\in \Hh^\alpha 
		\\
	D_{\sqrt{\rho_\alpha}}\xi & :=   (\xi_\bk/ \sqrt{\rho_{\alpha,\bk}})_{\bk\in \NNd_0}\,  , \qquad \xi=(\xi_\bk)_{\bk\in \NNd_0}
		\\
		(B\xi)(\bx)& :=  \sum_{\bk\in \NNd_0} \xi_\bk\,  H_\bk(\bx)\, , \qquad \bx \in  \RRd\, .
	\end{aligned}
\end{equation*}
It is obvious that $\|A\|=\|B\|=1$.
From the ideal property (c) and the identity $I_\gamma = B\,D_{\sqrt{\rho_{\alpha}}} \, A$ it follows
\[
s_n\big(I_\gamma\big) \le s_n \big(D_{\sqrt{\rho_{\alpha}}} \big) \, .
\]
 
It is easy to see that the operators $A$ and $B$ are invertible and  that $\|A^{-1}\|=\|B^{-1}\|=1$. 
Employing the same type of arguments with respect to the diagram
	\begin{equation*}
	\begin{CD}
		\Hh^\alpha  @ > I_\gamma >> L_2(\RRd,\gamma) \\
@AA {A^{-1}} A @VV B^{-1} V\\
		\ell_2(\NNd_0) @ > D_{\sqrt{\rho_{\alpha}}} >> \ell_2(\NNd_0) \, 
	\end{CD}
\end{equation*} we also get
\[
s_n\big(I_\gamma\big) \geq s_n \big(D_{\sqrt{\rho_{\alpha}}}\big) \, .
\]
Consequently, we obtain
\[
s_n\big(I_\gamma\big)=
s_n\big(D_{\sqrt{\rho_{\alpha}}}\big) .
\]

Note, that the sequence $s_n\big(I_\gamma\big)=s_n\big(D_{\sqrt{\rho_{\alpha}}})$ is the non-increasing rearrangement of the sequence
$(1/\sqrt{\rho_{\alpha,\bk}}
)_{\bk\in \NNd_0}$, see \cite[Theorem 11.11.3]{Pie80B}. 
We put
\[
c(r,d):= \bigg|\bigg\{\bk \in \NNd_0: ~\prod_{j=1}^d (k_j+1)
\leq r \bigg\}\bigg|, \quad r \in \NN \, .
\]
From  \cite[Theorems 3.4 and 3.5]{ChD16} (with $a=1$) we find
\begin{equation}\label{ch}
\frac{1}{(d-1)!}\,  \frac{r(\ln r)^d}{\ln r+d} < c(r,d) < \frac{1}{(d-1)!} \, \frac{r(\ln r +d\ln 2)^d}{\ln r+d\ln 2 +d-1}
\end{equation}
for $r>r*>1$.  
For $n\geq 2$, choose $r\in \NN$ such that $ c(r-1,d) <n \le c(r,d)$. By the definition of $c(r,d)$ we  have
$$
r^{-\alpha/2}= s_{c(r,d)}  (I_\gamma) \leq  s_n  (I_\gamma) \leq  s_{c(r-1,d)}  (I_\gamma) =(r-1)^{-\alpha/2}.
$$
Clearly, $\lim_{r\to \infty} c(r, d) = \infty$. Moreover, the sequence $n\, (\ln n)^{-(d-1)}$ is increasing for 
$n > e^{d-1}$. Hence, we obtain
for sufficiently large $r \in  \NN$ the two-sided inequality
\[
\frac{ c(r - 1, d)}{r (\ln c(r-1,d))^{d-1}}
\le   \, \frac{s_n(I_\gamma)^{2/\alpha}}{n^{-1}(\ln n)^{d-1}}\le 
\frac{c(r, d)}{(r-1)\, (\ln c(r, d))^{d-1}}\, .
\]
Applying \eqref{ch} the claim follows. 
\hfill
\end{proof}
\begin{remark} Observe that, we do not have factor $2^d$ in the asymptotic constant in \eqref{eq-asymptotic}. This is a difference compared to the asymptotic constant of the embedding    $I: \tilde{W}_{2}^\alpha(\IId) \to \title{L}_2(\IId)$, see \cite{KSU15}.
\end{remark}


Let us emphasize that for  $\alpha>0$, the space $\Hh^\alpha$ is not embedded in $L_q(\RRd)$ with $2<q\leq \infty$. Therefore, in the next step we study the asymptotic behavior of pseudo $s$-numbers of the embedding of $\Hh^\alpha$ into the space $L_{\infty}^{\sqrt{g}}(\RRd)$, where the norm of $f\in L_{\infty}^{\sqrt{g}}(\RRd)$ is defined by
$$
\|f\|_{L_{\infty}^{\sqrt{g}}(\RRd)}: = \sup_{\bx\in \RRd}\big|f(\bx)\sqrt{g(\bx)}\big|<\infty.
$$
We have the following theorem.
\begin{theorem}
Let $\alpha>1-\frac{1}{6}$ and $s\in \{a,b,c,d,e,x\}$. Then we have
\begin{equation*}
	\begin{aligned}
n^{-\frac{\alpha}{2}-\frac{d}{4}} (\log n)^{(\frac{\alpha}{2}+\frac{d}{4})(d-1)} & \ll
s_n\big(I_\gamma: \Hh^\alpha\to L_{\infty}^{\sqrt{g}}(\RRd)\big)
\\
&\ll n^{-\frac{\alpha}{2}-\frac{1}{12}+\frac{1}{2}}(\log n)^{(\frac{\alpha}{2}+\frac{1}{12})(d-1)}\,, \ n\to \infty.
	\end{aligned}
\end{equation*}
 
\end{theorem}
\begin{proof} 
{\bf Step 1.}  
We first show that the embedding  $I_\gamma:{\Hh}^\alpha\to  L_\infty^{\sqrt{g}}(\RRd)$ is continuous. Recall an inequality in \cite{DILP18}
$$
H_\bk(\bx)\sqrt{g(\bx)}\leq \prod_{j=1}^d \min\bigg(1,\frac{\sqrt{\pi}}{k_j^{1/12}}\bigg)\asymp \prod_{j=1}^d\frac{1}{(1+k_j)^{1/12}} ,\ \ \bx\in \RRd.
$$ If $f=\sum_{\bk \in \NNd_0}\hat{f}(\bk)H_\bk \in \Hh^\alpha$, then we have
\begin{equation}\label{eq-Cauchy}
\begin{aligned}
\Bigg\|\sum_{\bk \in \NNd_0}\hat{f}(\bk)H_\bk\Bigg\|_{L_\infty^{\sqrt{g}}(\RRd)}
& 
=
\Bigg\|\sqrt{g(\bx)}\sum_{\bk\in \NNd_0}  \hat{f}(\bk)     H_\bk(\bx)\Bigg\|_{L_\infty(\RRd)}
\\
&
\leq \sum_{\bk\in \NNd_0}  |\hat{f}(\bk)| \sup_{\bx\in \RRd} \big|H_\bk(\bx)\sqrt{g(\bx)}\big|
\\
&
\ll \sum_{\bk\in \NNd_0}  |\hat{f}(\bk)|  \prod_{j=1}^d\frac{1}{(1+k_j)^{1/12}}
\\
&
\leq \Bigg(\sum_{\bk\in \NNd_0} \rho_{\alpha,\bk} |\hat{f}(\bk)|^2\Bigg)^{1/2} \Bigg(\sum_{\bk\in \NNd_0}\frac{1}{\rho_{\alpha,\bk}} \prod_{j=1}^d\frac{1}{(1+k_j)^{1/6}}
\Bigg)^{1/2},
\end{aligned}
\end{equation}
where in the last estimate we used the Cauchy–Schwarz inequality. 
Assumption  $\alpha>1-\frac{1}{6}$ implies that
\begin{align*}
\Bigg(\sum_{\bk\in \NNd_0}\frac{1}{\rho_{\alpha,\bk}} \prod_{j=1}^d\frac{1}{(1+k_j)^{1/6}}
\Bigg)^{1/2} 
&
= \Bigg(\sum_{\bk\in \NNd_0}  \prod_{j=1}^d\frac{1}{(1+k_j)^{1/6+\alpha}}
\Bigg)^{1/2} 
\\
&= \Bigg(    \sum_{k=0}^\infty\frac{1}{(1+k)^{1/6+\alpha}}
\Bigg)^{d/2} <\infty.
\end{align*}
Consequently, we get 
$$
\|f(\bx)\|_{L_\infty^{\sqrt{g}}(\RRd)} \leq C \Bigg(\sum_{\bk\in \NNd_0} \rho_{\alpha,\bk} |\hat{f}(\bk)|^2\Bigg)^{1/2} =C\|f\|_{\Hh^\alpha}
$$
which means the continuous embedding of $\Hh^\alpha$ into $L_\infty^{\sqrt{g}}(\RRd)$.
\\
\noindent
{\bf Step 2.} {\it Upper bound.} Since approximation number is the largest pseudo $s$-number in the set $\{a,b,c,d,e,x\}$, we prove the upper bound for approximation numbers. Let $\bs = (s_1, \ldots, s_d)$ be a vector whose coordinates are nonnegative integers. 
We define
\begin{equation}\label{eq-Qxi}
\begin{aligned}
\delta(\bs)&:=\big\{\bk\in \NNd_0: \lfloor 2^{s_j-1}\rfloor \leq k_j\leq 2^{s_j}, j=1,\ldots,d\big\}
\\
Q_\xi&:=\bigcup_{|\bs|_1\leq \xi}\delta(\bs).
\end{aligned}
\end{equation}
For any $\xi \in \NN$, denote $A_\xi: \Hh^\alpha\to L_\infty^{\sqrt{g}}(\RRd)$  the operator defined by
$$
A_\xi(f):=\sum_{\bk\in Q_\xi} \hat{f}(\bk)H_\bk,\ \ f=\sum_{\bk \in \NNd_0}\hat{f}(\bk)H_\bk \in \Hh^\alpha.
$$
It is well-known that the cardinality  of $Q_\xi$ is equivalent to 
$
 2^\xi \xi^{d-1}
$, see, e.g., \cite[Section 2.3]{DTU18B}, 
which implies that ${\rm rank}(A_\xi)\asymp 2^\xi \xi^{d-1}$. As \eqref{eq-Cauchy} we have
\begin{align*}
\|f-A_\xi(f)\|_{L_\infty^{\sqrt{g}}(\RRd)}
&
=\Bigg\|\sum_{\bk\not \in Q_\xi}\hat{f}(\bk)H_\bk\Bigg\|_{L_\infty^{\sqrt{g}}(\RRd)}
\\
&
\ll  \Bigg(\sum_{\bk\not \in Q_\xi} \rho_{\alpha,\bk} |\hat{f}(\bk)|^2\Bigg)^{1/2} \Bigg(\sum_{\bk\not \in Q_\xi} \prod_{j=1}^d\frac{1}{(1+k_j)^{\alpha+1/6}}
\Bigg)^{1/2}.
\end{align*}
Observe that
\begin{align*}
\sum_{\bk\not \in Q_\xi} \prod_{j=1}^d\frac{1}{(1+k_j)^{\alpha+1/6}}
& = \sum_{m=\xi+1}^\infty \sum_{ |\bs|_1=m}\sum_{\bk\in \delta(\bs)}\prod_{j=1}^d\frac{1}{(1+k_j)^{\alpha+1/6}}
\\
& \ll \sum_{m=\xi+1}^\infty \sum_{|\bs|_1=m} \frac{2^m}{2^{m(\alpha+1/6)}}
\\
&
\ll \sum_{m=\xi+1}^\infty 2^{m(-\alpha-\frac{1}{6}+1)}m^{d-1}
\\
& 
\asymp 2^{\xi(-\alpha-\frac{1}{6}+1)}\xi^{d-1}.
\end{align*}
Therefore
$$
\|f-A_\xi(f)\|_{L_\infty^{\sqrt{g}}(\RRd)} \ll 2^{\xi(-\alpha-\frac{1}{6}+1)}\xi^{d-1} \|f\|_{\Hh^\alpha}
.$$
For $n\in \NN$, $n>2$, we choose $\xi=\xi(n) $ such that 
$
n\asymp 2^\xi \xi^{d-1}
$
which implies 
$
\xi \asymp \log n$ and $ 2^\xi \asymp \frac{n}{(\log n)^{d-1} }$. From this we get
\begin{align*}
	\|f-A_{\xi(n)}(f)\|_{L_\infty^{\sqrt{g}}(\RRd)}
&\ll \bigg(\frac{n }{(\log n)^{d-1}}\bigg)^{-\frac{\alpha}{2}-\frac{1}{12}+\frac{1}{2}} (\log n)^{\frac{d-1}{2}} \|f\|_{\Hh^\alpha}
\\
&
=
n^{-\frac{\alpha}{2}-\frac{1}{12}+\frac{1}{2}}(\log n)^{(d-1)(\frac{\alpha}{2}+\frac{1}{12})} \|f\|_{\Hh^\alpha}.
\end{align*}
By the definition of approximation numbers, we obtain the upper bound. 
\\
\noindent
{\bf Step 3.} {\it Lower bound.}
The inequalities \eqref{eq-inequality1}, \eqref{eq-inequality2} and \eqref{eq-inequality3} induce that Bernstein number is the smallest pseudo $s$-number in the set $\{a,b,c,d,e,x\}$. Therefore, it is enough to prove the lower bound for Bernstein numbers. For $\xi\in \NN$, let  $L(Q_\xi)$ be the subspace of $\Hh^\alpha$ defined by
$$
L(Q_\xi)=\Bigg\{ f=\sum_{\bk\in Q_\xi} a_\bk  H_\bk ,\  a_\bk \in \RR\Bigg \},
$$
where $Q_\xi$ is given in 
\eqref{eq-Qxi}. Denote $n:=n(\xi)={\rm dim}(L(Q_\xi))$. 
By the definition of Bernstein numbers we get
\begin{equation*}
	\begin{aligned}
		b_n(I_\gamma)\geq   \inf_{\substack{f\in L(Q_\xi)
				\\ f\not =0}} \dfrac{\|f\|_{L_\infty^{\sqrt{g}}(\RRd)}}{\| f\|_{\Hh^\alpha}} &=\inf_{\substack{f\in L(Q_\xi)
				\\ f\not =0}}\frac{\big\|\sqrt{g}\sum_{\bk \in Q_\xi}a_\bk H_\bk\big\|_{L_\infty(\RRd)}}{\big(\sum_{\bk \in Q_\xi} \rho_{\alpha,\bk}|a_\bk |^2\big)^{1/2}}.
	\end{aligned}
\end{equation*}

Denote by $a_m$ the $m$th Mhaskar-Rakhmanov-Saff number corresponding to $\sqrt{g}$. See \cite[Page 11]{Lu07} for a definition of this number. From \cite[Page 11]{Lu07}  we have  
\begin{equation*} \label{a_m(g)}
	a_m 
	\ = \
	\sqrt{m}.
\end{equation*}
For any polynomial $\varphi$ of degree $\le m$ we have the Nikol'skii-type inequality \cite[Theorem 9.1, Page 61]{Lu07} 
	\begin{equation} \label{eq-Nilkolski}
		\|\varphi \sqrt{g} \|_{L_2( \RR)}
		\ \le \
		C  a_m^{\frac{1}{2}}\, \|\varphi \sqrt{g} \|_{L_\infty( \RR)}\leq 	C  m^{\frac{1}{4}}\, \|\varphi \sqrt{g} \|_{L_\infty( \RR)}.
	\end{equation}
Observe that, if $f\in L(Q_\xi)$ then $f$ is a polynomial of degree $\leq$ $2^\xi$ with respect to each variable $x_j$, $j=1,\ldots,d$. 
Applying the inequality \eqref{eq-Nilkolski} continuously with respect to $x_j$, $j=1,\ldots,d$,  we get	
$$
2^{\frac{d\xi}{4}}\Bigg\|\sqrt{g}\sum_{\bk \in Q_\xi}a_\bk H_\bk\Bigg\|_{L_\infty(\RRd)} \gg 2^{\frac{d\xi}{4}}\Bigg\|\sqrt{g}\sum_{\bk \in Q_\xi}a_\bk H_\bk\Bigg\|_{L_2(\RRd)}=2^{\frac{d\xi}{4}}\Bigg\|\sum_{\bk \in Q_\xi}a_\bk H_\bk\Bigg\|_{L_2(\RRd,\gamma)}\,.
$$
Consequently, we find
\begin{equation*}
	\begin{aligned}
 b_n(I_\gamma)   
	&
\gg \inf_{(a_\bk)_{\bk\in Q_\xi}} \frac{2^{-\frac{d\xi}{4}}\big\|\sum_{\bk \in Q_\xi}a_\bk H_\bk\big\|_{L_2(\RRd,\gamma)}}{\big(\sum_{\bk \in Q_\xi} \rho_{\alpha,\bk}|a_\bk |^2\big)^{1/2}} 
\\
&
= \inf_{(a_\bk)_{\bk\in Q_\xi}}\frac{2^{-\frac{\xi d}{4}}\big(\sum_{\bk \in Q_\xi} |a_\bk |^2\big)^{1/2}}{2^{\frac{\alpha}{2}\xi}\big(\sum_{\bk \in Q_\xi} \frac{\rho_{\alpha,\bk}}{2^{\alpha\xi }}|a_\bk |^2\big)^{1/2}}\,.
	\end{aligned}
\end{equation*}
Since $\rho_{\alpha,\bk}\ll 2^{\alpha \xi}$ for $\bk\in Q_\xi$, we obtain
$$
b_n(I_\gamma)\gg 2^{-\frac{\alpha\xi}{2}-\frac{d\xi }{4}}.
$$
Finally from the equivalence $n\asymp 2^\xi \xi^{d-1}$ we get 
$$
b_n(I_\gamma) \gg n^{-\frac{\alpha}{2}-\frac{d}{4}} (\log n)^{(\frac{\alpha}{2}+\frac{d}{4})(d-1)}.
$$
The proof is competed.
\hfill
\end{proof}

\bibliographystyle{abbrv}

\bibliography{AllBib}
\end{document}